\documentclass[11pt]{amsart}
\usepackage{graphicx}

\usepackage{amsmath, amsfonts, amssymb, amsthm, hyperref, mathtools, mathrsfs}
\usepackage[table]{xcolor}
\usepackage{fullpage}

\newtheorem{theorem}{Theorem}[section]

\newtheorem{corollary}[theorem]{Corollary}

\newtheorem{proposition}[theorem]{Proposition}

\numberwithin{theorem}{section}
\numberwithin{equation}{section}

\renewcommand{\geq}{\geqslant}
\renewcommand{\leq}{\leqslant}

\begin{document}

\title{A generalization of Franklin's partition identity and a Beck-type companion identity}

\author{Gabriel Gray}
\address{University of Dayton}
\email{grayg1@udayton.edu}

\author{David Hovey}
\address{Princeton University}
\email{dh7965@princeton.edu}

\author{Brandt Kronholm}
\address{University of Texas Rio Grande Valley}
\email{brandt.kronholm@utrgv.edu}

\author{Emily Payne}
\address{University of Texas Rio Grande Valley}
\email{emily.payne01@utrgv.edu}

\author{Holly Swisher}
\address{Oregon State University}
\email{swisherh@oergonstate.edu}

\author{Ren Watson}
\address{University of Texas at Austin}
\email{renwatson@utexas.edu}

\thanks{The first, fourth, fifth, and sixth authors were supported by NSF grant DMS-2101906.}

\dedicatory{Dedicated to George Andrews and Bruce Berndt in celebration of their 85th birthdays.}



\begin{abstract}
Euler's classic partition identity states that the number of partitions of $n$ into odd parts equals the number of partitions of $n$ into distinct parts. We develop a new generalization of this identity, which yields a previous generalization of Franklin as a special case, and prove an accompanying Beck-type companion identity. 
\end{abstract}

\maketitle

\section{Introduction and Statement of Results}

A \textit{partition} of a positive integer $n$ is a finite non-increasing sequence of positive integers called \emph{parts} that sum to $n$.  We also define the empty partition to be the unique partition of $0$.  The partition function $p(n)$ counts the number of partitions of $n\geq 0$, and we write $p(n\mid *)$ for the function which counts the number of partitions of $n$ satisfying condition $*$. 

A fundamental partition identity due to Euler states that for any integer $n\geq 0$, the number of partitions of $n$ into odd parts is equal to the number of partitions of $n$ into distinct parts.  We can write this as
\begin{equation} \label{odd=distinct}
    p(n\mid\text{no part is divisible by $2$})=p(n\mid\text{no part appears $\geq 2$ times}).
\end{equation}
Glaisher's theorem \cite{glaisher1883} generalizes \eqref{odd=distinct} giving that for all $n\geq 0$ and $k \geq 1$, 
\begin{equation} \label{glaishersthm}
    p(n\mid\text{no part divisible by }k)=p(n\mid\text{no part appears }\geq k \text{ times}).
\end{equation}
Define the functions $O_{j,k}(n)$ to count the number of partitions of $n$ with exactly $j$ different parts divisible by $k$ (which can repeat) and $D_{j,k}(n)$ to count the number of partitions of $n$ with exactly $j$ different parts that appear at least $k$ times.  Then we can rewrite \eqref{odd=distinct} as $O_{0,2}(n) = D_{0,2}(n)$ and \eqref{glaishersthm} as $O_{0,k}(n) = D_{0,k}(n)$.

In 1883, Franklin \cite{franklin1883} further generalized \eqref{glaishersthm} by proving that for all $n,j\geq 0$ and $k \geq 1$, 
\begin{equation} \label{franklinseq}
    O_{j,k}(n)=D_{j,k}(n).
\end{equation}

Despite the fact that Euler's result dates to the middle of the $18^{\text{th}}$ century and those of Glaisher, and Franklin are well over one hundred years old, there is always more to discover.  For example, in 2022, Amdeberhan, Andrews, and Ballantine \cite{AAB} proved the interesting refinement $O_{1,k}^{(u)}(n) = D_{1,k}^{(u)}(n)$, where $O_{1,k}^{(u)}(n)$ is the number of partitions counted by $O_{1,k}^{(u)}(n)$ such that the one part divisible by $k$ is repeated exactly $u$ times, and $D_{1,k}^{(u)}(n)$ is the number of partitions counted by $D_{1,k}^{(u)}(n)$ such that the one part repeated at least $k$ times is $u$.  We note that the following identity, discovered recently by the second author, can also be considered as a generalization of Euler's and Glaisher's identities.  For all integers $n\geq 0$, $k \geq 1$, and $b \geq 1$,
\begin{multline} \label{hoveysthm}
    p(n\mid\text{no part is divisible by }kb)\\=p(n\mid\text{no part is both divisible by }b \text{ and appears }\geq k \text{ times}).
\end{multline}
Then \eqref{glaishersthm} follows from \eqref{hoveysthm} by setting $b=1$. 

In an email correspondence, with the third author, George Andrews remarked {\em ``...I do not remember seeing this before. ... I believe Grahm Lord (PhD, Temple, early 1970s) had a paper containing many theorems of this nature; unfortunately, I have not been able to track it down."} \cite{AndrewsKronholmemail}

We establish identity \eqref{hoveysthm} from the associated generating functions.  Define the $q$-Pochhammer symbol by $(a;q)_n = (1-a)(1-q)(1-q^2)\cdots (1-q^{n-1})$, where $n\leq \infty$.  Then since 
\[
\prod_{n\geq 0} \frac{1}{(1-q^{mn+1})(1-q^{mn+2})\cdots (1-q^{mn+m-1})} = \frac{(q^m;q^m)_\infty}{(q;q)_\infty} = \prod_{n\geq 1} (1+q^n + q^{2n} + \cdots q^{(m-1)n}),
\]
we can interpret $\frac{(q^m;q^m)_\infty}{(q;q)_\infty}$ as generating partitions into parts not divisible by $m$, and also as generating partitions where each part appears at most $m-1$ times.  Thus \eqref{hoveysthm} follows from the fact that 
\[
\frac{(q^{kb};q^{kb})_\infty}{(q;q)_\infty} = \frac{(q^{b};q^{b})_\infty}{(q;q)_\infty}  \cdot \frac{(q^{kb};q^{kb})_\infty}{(q^b;q^b)_\infty}, 
\]
since the right hand side generates partitions where parts not divisible by $b$ are unrestricted, and parts divisible by $b$ appear at most $k-1$ times.
 
We can write (\ref{odd=distinct}), (\ref{glaishersthm}), (\ref{franklinseq}), and (\ref{hoveysthm}) in a more uniform notation by introducing the parameter $b$.  Let $O_{j,k,b}(n)$ count the number of partitions of $n$ with exactly $j$ parts divisible by $kb$ (repetitions allowed) and $D_{j,k,b}(n)$ count the number of partitions of $n$ with exactly $j$ parts that are both divisible by $b$ and appear $\geq k$ times. We also let $\mathcal{O}_{j,k,b}(n)$ and $\mathcal{D}_{j,k,b}(n)$ denote the set of all partitions counted by $O_{j,k,b}(n)$ and $D_{j,k,b}(n)$, respectively, so that $$O_{j,k,b}(n)=|\mathcal{O}_{j,k,b}(n)|~ \text{ and } ~D_{j,k,b}(n)=|\mathcal{D}_{j,k,b}(n)|.$$

With this notation \eqref{glaishersthm} becomes the statement that $O_{j,k,1}(n)=D_{j,k,1}(n)$ and \eqref{hoveysthm} becomes the statement $O_{0,k,b}(n)=D_{0,k,b}(n)$.  Thus it is natural to ask whether the equality holds for all choices of $j,k,b,n$.  We show in our first result that this is indeed the case. 

\begin{theorem} \label{OD24}
    For all integers $n,j \geqslant 0$, and $k,b \geqslant 1$, $$O_{j,k,b}(n)=D_{j,k,b}(n).$$
\end{theorem}

In 2017, George Beck \cite{beck2017oeis} conjectured an interesting identity related to Euler's identity $O_{0,2}(n) = D_{0,2}(n)$.   He conjectured that for any positive integer $n$, if you count the total number of parts over all partitions counted by $O_{0,2}(n)$, and then subtract the total number of parts over all partitions counted by $D_{0,2}(n)$, that you will get $O_{1,2}(n)=D_{1,2}(n)$.  Thus $O_{1,2}(n)=D_{1,2}(n)$ could be considered to count the excess in the total number of parts of $O_{0,2}(n)$ over $D_{0,2}(n)$.

For example, the partitions counted by $O_{0,2}(5)$ are $5, 3+1+1, 1+1+1+1+1$ which have $9$ total parts.  And the partitions counted by $D_{0,2}(5)$ are $5, 4+1, 3+2$ which have $5$ total parts.  Then $9-5=4=O_{1,2}(5)$, as there are four partitions counted by $O_{1,2}(5)$, namely $4+1, 3+2, 2+2+1, 2+1+1+1$. 

We use the following notation to describe this phenomenon.  Let $\ell(\pi)$ count the number of parts in the partition $\pi$.  Given two partition counting functions, $a(n)=p(n\mid *)$ and $b(n)=p(n\mid \heartsuit)$, with $A(n)$ and $B(n)$ the sets of partitions counted by $a(n)$ and $b(n)$, respectively.  Define 
\begin{equation}\label{def:E(a,b)}
E(a(n),b(n)) = \sum_{\pi \in A(n)} \!\!\! \ell(\pi) - \!\!\! \sum_{\pi \in B(n)} \!\!\! \ell(\pi),
\end{equation}
so that $E(a(n),b(n))$ counts excess in the total number of parts in all partitions counted by $a(n)$ over the total number of parts in all partitions counted by $b(n)$.  Then Beck's conjecture states that for all $n\geq 1$,
\begin{equation}\label{Beck1}
E(O_{0,2}(n),D_{0,2}(n)) = O_{1,2}(n).
\end{equation}
This striking conjecture was proved by Andrews \cite[Thm. 1]{andrews2017eulers} later that year.  

Beck also conjectured a second identity which was proved by Andrews \cite[Thm. 2]{andrews2017eulers}.  If we define $\bar{\ell}(\pi)$ to count the number of different parts occurring in the partition $\pi$, then Andrews proved that for all $n\geq 1$,
\[
\sum_{\pi \in \mathcal{D}_{0,2}(n)} \!\!\! \ell(\pi) - \!\!\! \sum_{\pi \in \mathcal{O}_{0,2}(n)} \!\!\! \bar{\ell}(\pi) = D_{1,3}(n) = O_{1,3}(n).
\]

Given a partition identity $a(n)=b(n)$, an identity which counts the excess in number of parts (perhaps with certain conditions) for $a(n)$ over $b(n)$ is now called a Beck-type companion identity.  Beck-type identities have also been established for different types of identities, for example Ballantine and Folsom \cite{BallantineFolsom} established Beck-type identities for the Rogers-Ramanujan functions.

Fu and Tang \cite[Thm. 1.5]{futang2017generalizing} generalized \eqref{Beck1} as follows.  Define $\ell_{m,n}(\pi)$ to count the number of parts $\equiv n \pmod{m}$ in $\pi$.  For all $n \geq 0$ and $k \geq 2$, they prove that
\begin{equation*}
\sum_{\pi \in \mathcal{O}_{0,k}(n)} \!\!\! \ell_{k,1}(\pi) - \!\!\! \sum_{\pi \in \mathcal{D}_{0,k}(n)} \!\!\! \bar{\ell}(\pi) = O_{1,k}(n) = D_{1,k}(n).
\end{equation*}
When $k=2$, all parts in $\pi \in \mathcal{O}_{0,2}(n)$ are odd so are counted, and all parts in $\pi \in \mathcal{D}_{0,2}(n)$ are distinct so are counted.  Thus this recovers \eqref{Beck1}.

 
In 2021, Ballantine and Welch \cite[Thm. 3]{ballantinewelch2021becktype} proved a more general Beck-type companion identity for Franklin's identity \eqref{franklinseq}. They prove that for all $n\geq 1$, $j\geq 0$, and $k \geq 2$, 
\begin{equation} \label{bwelchbeckid}
    \frac{1}{k-1} E(O_{j,k}(n),D_{j,k}(n))=(j+1)O_{j+1,k}(n)-jO_{j,k}(n)=(j+1)D_{j+1,k}(n)-jD_{j,k}(n).
\end{equation} 
When $j=0$ this fully generalizes \eqref{Beck1} for arbitrary $k\geq 2$.  Ballantine and Welch \cite{ballantinewelch2021becktype} observe that when $j\geq 1$, the right hand side of \eqref{bwelchbeckid} is nonpositive.  However, an immediate corollary follows from \eqref{bwelchbeckid} for which this is not the case.  Define $O_{\leq j,k}(n)=\sum_{i=0}^j O_{i,k}(n)$ which counts the number of partitions of $n$ with at most $j$ parts divisible by $k$, and $D_{\leq j,k}(n)=\sum_{i=0}^j D_{i,k}(n)$ which counts the number of partitions of $n$ with at most $j$ parts that occur at least $k$ times.  Then $E(O_{\leq j,k}(n), D_{\leq j,k}(n)) = \sum_{i=0}^j E(O_{i,k}(n),D_{i,k}(n))$, so from \eqref{bwelchbeckid} it follows that for all $n\geq 1$, $j\geq 0$, $k\geq 2$,
\begin{equation} \label{bwelchcor}
E(O_{\leq j,k}(n), D_{\leq j,k}(n)) = (k-1)(j+1)O_{j+1,k}(n).
\end{equation} 

Our next result is a Beck-type companion identity for Theorem \ref{OD24} that generalizes \eqref{bwelchbeckid}. 

\begin{theorem} \label{Ojkb/Djkb-beck}
For all $n,b \geq 1$, $j\geq 0$, and $k \geq 2$, 
\[
\frac{1}{k-1}E(O_{j,k,b}(n),D_{j,k,b}(n))=(j+1)O_{j+1,k,b}(n)-jO_{j,k,b}(n)=(j+1)D_{j+1,k,b}(n)-jD_{j,k,b}(n).
\]
\end{theorem} 

From Theorem \ref{OD24} we obtain the following immediate corollary which generalizes \eqref{bwelchcor}, where $O_{\leq j,k,b}(n) = \sum_{i=0}^j O_{i,k,b}(n)$ and $D_{\leq j,k,b}(n) = \sum_{i=0}^j D_{i,k,b}(n)$.

\begin{corollary} \label{corollary}
For all $n\geq 1$, $j\geq 0$, and $k,b\geq 1$,
\[
E(O_{\leq j,k,b}(n), D_{\leq j,k,b}(n)) = (k-1)(j+1)O_{j+1,k,b}(n).
\]
\end{corollary}

We now outline the remainder of this paper. In Section \ref{sec:FHgen}, we provide two proofs of Theorem \ref{OD24}, one via generating functions and the other via a combinatorial bijection.  In Section \ref{sec:Beck} we provide two proofs of Theorem \ref{Ojkb/Djkb-beck}, one directly using differentiation of generating functions and the other using a refinement approach of Ballantine and Welch \cite[Thm. 4]{ballantinewelch2021becktype}.  

\section{Generating Functions} 

Here we observe a few of the fundamental building blocks that we will need in the construction of our generating functions.  First, we observe that
\begin{equation}\label{eq:ndivtotal}
\frac{1}{(wq^n;q^n)_\infty} \mbox{ and } \frac{(wq^n;q^n)_\infty}{(wq;q)_\infty}
\end{equation}
generates partitions into parts $\equiv 0 \pmod{n}$ and $\not\equiv 0 \pmod{n}$, respectively, with the power of $q$ tracking the size of the partition, and the power of $w$ tracking the total number of parts appearing in these partitions.  Letting $w=1$ generates the same without tracking the number of parts.

Meanwhile, we can interpret
\begin{equation}\label{eq:divdiff}
\frac{((1-z)wq^n;q^n)_\infty}{(wq^n;q^n)_\infty} = \prod_{m\geq 1} \left( 1+\frac{zwq^{nm}}{1-wq^{nm}} \right) = \prod_{m\geq 1} (1+zwq^{nm}+zw^2q^{2nm}+zw^3q^{3nm}+\cdots)
\end{equation}
as generating partitions into parts $\equiv 0 \pmod{n}$, with the power of $q$ tracking the size of the partitions, the power of $z$ tracking the number of different parts appearing in these partitions, and the power of $w$ tracking the total number of parts.  Again, letting $w=1$ generates the same without tracking the total number of parts.

Since
\begin{equation*}\label{eq:q-intobs}
(1-wq^n)(1+wq^n+w^2q^{2n}+\cdots +w^{k-1}q^{(k-1)n}) = 1-w^kq^{kn},
\end{equation*}
it follows that 
\begin{multline*}\label{eq:obs2}
\frac{1-(1-z)w^kq^{kn}}{1-wq^n} =  \frac{1-w^kq^{kn}}{1-wq^n} \cdot \frac{1-(1-z)w^kq^{kn}}{1-w^kq^{kn}} 
= \frac{1-w^kq^{kn}}{1-wq^n} \left( 1+\frac{zw^kq^{kn}}{1-w^kq^{kn}} \right) \\
= (1+wq^n+w^2q^{2n}+\cdots +w^{k-1}q^{(k-1)n})\cdot(1+zw^kq^{kn}+zw^{2k}q^{2kn}+zw^{3k}q^{3kn}+\cdots) \\
= 1+wq^n+w^2q^{2n}+\cdots +w^{k-1}q^{(k-1)n} +zw^kq^{kn}+zw^{k+1}q^{(k+1)n}+zw^{k+2}q^{(k+2)n}+\cdots .
\end{multline*}
Thus we can interpret
\begin{equation}\label{eq:ktimes}
\frac{((1-z)w^kq^{kn};q^{kn})_\infty}{(wq^n;q^n)_\infty}
\end{equation}
as generating partitions into parts $\equiv 0 \pmod{n}$, with the power of $q$ tracking the size of the partitions, the power of $z$ tracking the number of different parts occurring at least $k$ times in these partitions, and the power of $w$ tracking the total number of parts.  Again, letting $w=1$ generates the same without tracking the total number of parts.

We can now quickly deduce Theorem \ref{OD24} by showing that $O_{j,k,b}(n)$ and $D_{j,k,b}(n)$ have the same generating functions.

\begin{proof}[Proof of Theorem \ref{OD24} via generating functions]

By the definition of $O_{j,k,b}(n)$, it follows from \eqref{eq:ndivtotal} and \eqref{eq:divdiff} that
\begin{equation}\label{eq:Ozq}
\sum_{n,j \geq 0} O_{j,k,b}(n)z^j q^n = \frac{(q^{kb};q^{kb})_\infty}{(q;q)_\infty} \cdot \frac{((1-z)q^{kb};q^{kb})_\infty}{(q^{kb};q^{kb})_\infty} = \frac{((1-z)q^{kb};q^{kb})_\infty}{(q;q)_\infty}.
\end{equation}

Moreover, by the definition of $D_{j,k,b}(n)$, it follows from \eqref{eq:ndivtotal} and \eqref{eq:ktimes} that
\begin{equation}\label{eq:Dzq}
\sum_{n,j \geq 0} D_{j,k,b}(n)z^j q^n = \frac{(q^b;q^b)_\infty}{(q;q)_\infty} \cdot  \frac{((1-z)q^{kb};q^{kb})_\infty}{(q^b;q^b)_\infty} = \frac{(1-z)q^{kb};q^{kb})_\infty}{(q;q)_\infty},
\end{equation}
so the two generating functions are equal.
\end{proof}

A related generating function that will be useful is obtained as follows.  We can interpret the coefficient of $z^jq^n$ in 
\begin{equation}\label{(j+1)Oint}
\frac{(q^{kb};q^{kb})_\infty}{(q;q)_\infty} \sum_{i\geq 1} \frac{q^{kbi}}{1-q^{kbi}}  \prod_{\substack{m\geq 1 \\ m\neq i}}\frac{1-(1-z)q^{kbm}}{1-q^{kbm}}
\end{equation} 
as $j+1$ times the number of partitions of $n$ consisting of exactly $j+1$ different parts that are divisible by $kb$, since the sum will equally count the partitions for each of the $j+1$ parts that are divisible by $kb$.  Thus, \eqref{(j+1)Oint} gives a generating function for $(j+1)O_{j+1,k,b}(n)$.  Moreover,
\[
\frac{(q^{kb};q^{kb})_\infty}{(q;q)_\infty} \sum_{i\geq 1} \frac{q^{kbi}}{1-q^{kbi}}  \prod_{\substack{m\geq 1 \\ m\neq i}}\frac{1-(1-z)q^{kbm}}{1-q^{kbm}} 
= \frac{((1-z)q^{kb};q^{kb})_\infty}{(q;q)_\infty} \sum_{i\geq 1} \frac{q^{kbi}}{1-(1-z)q^{kbi}},  
\]
So we have that
\begin{align}\label{(j+1)Ogen}
\sum_{n,j \geq 0} (j+1)O_{j+1,k,b}(n)z^j q^n 
&= \frac{((1-z)q^{kb};q^{kb})_\infty}{(q;q)_\infty} \sum_{i\geq 1} \frac{q^{kbi}}{1-(1-z)q^{kbi}} \\
&= \frac{(q^{kb};q^{kb})_\infty}{(q;q)_\infty} \sum_{i\geq 1} \frac{q^{kbi}}{1-q^{kbi}}  \prod_{\substack{m\geq 1 \\ m\neq i}}\frac{1-(1-z)q^{kbm}}{1-q^{kbm}}. \nonumber
\end{align} 

We will also need a few three variable generating functions.  Define $O_{j,k,b}(m,n)$ to be the number of partitions counted by $O_{j,k,b}(n)$ that have $m$ parts.  Let $D_{j,k,b}(m,n)$ be the number of partitions counted by $D_{j,k,b}(n)$ that have $m$ parts.  Then from \eqref{eq:ndivtotal}, \eqref{eq:divdiff}, and \eqref{eq:ktimes}, we have

\begin{align}
\sum_{n,j,m \geq 0} O_{j,k,b}(m,n)z^j w^m q^n & = \frac{(wq^{kb};q^{kb})_\infty}{(wq;q)_\infty} \cdot \frac{((1-z)wq^{kb};q^{kb})_\infty}{(wq^{kb};q^{kb})_\infty} = \frac{((1-z)wq^{kb};q^{kb})_\infty}{(wq;q)_\infty}, \label{eq:Ozwq} \\
\sum_{n,j,m \geq 0} D_{j,k,b}(m,n)z^j w^m q^n & = \frac{(wq^b;q^b)_\infty}{(wq;q)_\infty} \cdot  \frac{((1-z)w^kq^{kb};q^{kb})_\infty}{(wq^b;q^b)_\infty} = \frac{((1-z)w^kq^{kb};q^{kb})_\infty}{(wq;q)_\infty}. \label{eq:Dzwq}
\end{align}

In a different vein, when $0\leq t \leq k-1$, define $O_{j,k,b,t}(m,n)$ to be the number of partitions counted by $O_{j,k,b}(n)$ that have $m$ total parts $\equiv tb \pmod{bk}$.  Thus when $1\leq t \leq k-1$, so that $tb\not\equiv 0 \pmod{kb}$, it follows from \eqref{eq:ndivtotal} and \eqref{eq:divdiff} that

\begin{multline}\label{eq:Ozwq-t}
\sum_{n,j,m \geq 0} O_{j,k,b,t}(m,n)z^j w^m q^n = \frac{((1-z)q^{kb};q^{kb})_\infty}{(q^{kb};q^{kb})_\infty} \cdot \frac{(q^{kb};q^{kb})_\infty (q^{tb};q^{kb})_\infty}{(q;q)_\infty} \cdot \frac{1}{(wq^{tb};q^{kb})_\infty} \\
= \frac{((1-z)q^{kb};q^{kb})_\infty(q^{tb};q^{kb})_\infty}{(q;q)_\infty(wq^{tb};q^{kb})_\infty}.
\end{multline}
And when $t=0$, it follows from \eqref{eq:ndivtotal} and \eqref{eq:divdiff} that
\begin{equation}\label{eq:Ozwq-0}
\sum_{n,j,m \geq 0} O_{j,k,b,0}(m,n)z^j w^m q^n =  \frac{((1-z)wq^{kb};q^{kb})_\infty(q^{kb};q^{kb})_\infty }{(wq^{kb};q^{kb})_\infty(q;q)_\infty}.
\end{equation}

For our last generating function we need the following definitions.  For a partition $\pi$ we define $s_i(\pi)$ to be the \emph{multiplicity} of the part $i$ in $\pi$, i.e., the number of occurrences of $i$ in $\pi$.  And for a fixed positive integer $k$, we define $r_i(\pi)$ to be the residue of $s_i(\pi)$ modulo $k$, and call $r_i(\pi)$ the \emph{residual multiplicity modulo $k$} of the part $i$ in $\pi$.  Thus setting $u_i(\pi)=\lfloor \frac{s_i(\pi)}{k} \rfloor$, we have uniquely that
\begin{equation}\label{resmultdef}
s_i(\pi) = k u_i(\pi)+ r_i(\pi),
\end{equation}
with $0\leq r_i(\pi) \leq k-1$.  Observe that for any $i\geq 1$ and $0\leq t \leq k-1$, we can interpret a pair of choices from the product
\[
(1+q^{bi}+q^{2bi}+ \cdots + q^{(t-1)bi} + wq^{tbi} + wq^{(t+1)bi} + \cdots + wq^{(k-1)bi})(1+zq^{kbi}+zq^{2kbi} + \cdots)
\]
as determining a multiplicity for the part $bi$ in a partition $\pi$ in terms of $r_{bi}(\pi)$ and $u_{bi}(\pi)$.  Moreover,
\[
(1+q^{bi}+q^{2bi}+ \cdots + q^{(t-1)bi} + wq^{tbi} + wq^{(t+1)bi} + \cdots + wq^{(k-1)bi}) = \frac{(1-q^{tbi})+w(q^{tbi}-q^{kbi})}{1-q^{bi}},
\]
so we can interpret
\begin{equation}\label{resmult}
\frac{((1-z)q^{kb};q^{kb})_\infty}{(q^{kb};q^{kb})_\infty(q^{b};q^{b})_\infty} \prod_{i\geq 1}((1-q^{tbi})+w(q^{tbi}-q^{kbi}))
\end{equation}
as generating partitions into parts $\equiv 0 \pmod{b}$ where $q$ is tracking the size of $\pi$, $z$ is tracking the number of different parts $\equiv 0 \pmod{b}$ that appear in $\pi$ at least $k$ times, and $w$ is tracking the number of different parts $bi$ in $\pi$ with residual multiplicity at least $t$.

For $0\leq t \leq k-1$, we define $D_{j,k,b,t}(m,n)$ to be the number of partitions counted by $D_{j,k,b}(n)$ that have $m$ different parts $\equiv 0 \pmod{b}$ with residual multiplicity modulo $k$ at least $t$.  Then it follows from \eqref{eq:ndivtotal} and \eqref{resmult} that
\begin{equation}\label{eq:Dzwq-t}
\sum_{n,j,m \geq 0} D_{j,k,b,t}(m,n)z^j w^m q^n = \frac{((1-z)q^{kb};q^{kb})_\infty}{(q;q)_\infty(q^{kb};q^{kb})_\infty} \prod_{i\geq 1}((1-q^{tbi})+w(q^{tbi}-q^{kbi})).
\end{equation}

\section{Combinatorial Proof of Theorem \ref{OD24}} \label{sec:FHgen}
We now provide a combinatorial proof of Theorem \ref{OD24} by constructing a bijective correspondence between $\mathcal{O}_{j,k,b}(n)$ and $\mathcal{D}_{j,k,b}(n)$. 

\begin{proof}[Bijective proof of Theorem \ref{OD24}]
Since the result is trivially true when $k=1$, throughout we assume $k\geq 2$.

We will use the following notation.  For a partition $\pi$, we will write $\pi = 1^{s_1} 2^{s_2} \cdots \ell^{s_\ell}$ to denote that $s_i=s_i(\pi)$ the (nonnegative) multiplicity of $i$ in $\pi$.  As in \eqref{resmultdef}, we write $s_i = r_i + ku_i$, where $r_i=r_i(\pi)$ is the residual multiplicity modulo $k$ of the part $i$ in $\pi$.  We further break down $s_i$ into its base $k$ representation by writing $s_i = r_i + u_{i1}k + u_{i2}k^2 + \cdots + u_{iw}k^w$ with $0\leq u_{ij} \leq k-1$.  Lastly, when $i\equiv 0 \pmod{b}$, we define $\alpha_i\geq 0$ such that $i=bk^{\alpha_i}m_i$ where $k\nmid m_i$.

We now define a map $\varphi: \mathcal{O}_{j,k,b}(n) \rightarrow \mathcal{D}_{j,k,b}(n)$ as follows.  For $\pi \in \mathcal{O}_{j,k,b}(n)$, define $\varphi(\pi)$ by modifying the occurrences of each part $i$ of $\pi$ via
\[
i^{s_i} \mapsto 
\begin{cases} 
(\frac{i}{k})^{ks_i} & \text{ if } kb \mid i, \\
i^{r_{i}}, (k i)^{u_{i1}}, (k^2 i)^{u_{i2}}, \ldots, (k^w i)^{u_{iw}} & \text{ if } b\mid i, \, kb\nmid i, \\
i^{s_i} & \text{ if } b \nmid i.
\end{cases}
\]
To show that $\varphi(\pi) \in \mathcal{D}_{j,k,b}(n)$ we fist observe that $\varphi(\pi)$ is again a partition of $n$ as in each case, the size $s_i i$ of the parts $i^{s_i}$ remains unchanged.  We next observe that since $\pi \in \mathcal{O}_{j,k,b}(n)$, there are exactly $j$ parts of $\pi$ that are divisible by $kb$.  Thus from the first case of the definition of $\varphi$, we obtain $j$ different parts of $\varphi(\pi)$ that are both divisible by $b$ and occur at least $k$ times.  Parts arising from the other cases do not satisfy these conditions because in the second case each of the parts occur at most $k-1$ times, and in the third case parts are not divisible by $b$.  However, we need to check whether there could be overlap between the outputs of the second cases which combine to create a new part both divisible by $b$ and occurring at least $k$ times.  If in the second case we have $\alpha, \beta \geq 0$ and $h\neq i$ such that $k^\alpha h = k^\beta  i$, then $\alpha\neq \beta$ so without loss of generality, $\alpha>\beta$ and $k^{\alpha-\beta} h = i$. But then since $b\mid h$ we have $kb\mid i$ which is a contradiction.  Thus we conclude $\varphi(\pi) \in \mathcal{D}_{j,k,b}(n)$.

We next show that $\varphi$ is bijective by constructing an inverse.  Define the map $\psi: \mathcal{D}_{j,k,b}(n) \rightarrow \mathcal{O}_{j,k,b}(n)$ as follows.  For $\pi \in \mathcal{D}_{j,k,b}(n)$, define $\psi(\pi)$ by modifying the occurrences of each part $i$ of $\pi$ via
\[
i^{s_i} \mapsto 
\begin{cases} 
(ki)^{u_i}, (\frac{i}{k^{\alpha_i}})^{k^{\alpha_i}r_i} & \text{ if } b \mid i, \\
i^{s_i} & \text{ if } b \nmid i.
\end{cases}
\]
To show that $\psi(\pi) \in \mathcal{O}_{j,k,b}(n)$ we fist observe that $\psi(\pi)$ is again a partition of $n$ as in each case, the size $s_i i$ of the parts $i^{s_i}$ remains unchanged.  We next observe that since $\pi \in \mathcal{D}_{j,k,b}(n)$, there are exactly $j$ parts of $\pi$ that are both divisible by $b$ and occurring at least $k$ times.  Thus from the first output of the first case of the definition of $\psi$, we obtain $j$ different parts of $\psi(\pi)$ that are divisible by $kb$.  Moreover all additional parts created are not divisible by $kb$.  Thus  $\psi(\pi) \in \mathcal{O}_{j,k,b}(n)$.

To see that $\varphi \circ \psi (\pi) = \pi$ for any $\pi \in \mathcal{D}_{j,k,b}(n)$, observe that if $i$ is a part in $\pi$ divisible by $b$ with multiplicity $s_i=r_i+ku_i$, then $\psi(i^{s_i}) = (ki)^{u_i}, (\frac{i}{k^{\alpha_i}})^{k^{\alpha_i}r_i}$.  Since $kb \mid ki$, $kb\nmid \frac{i}{k^{\alpha_i}}$, and $k^{\alpha_i}r_i$ is already written base $k$, it follows that $\varphi((ki)^{u_i}) = i^{ku_i}$, and $\varphi((\frac{i}{k^{\alpha_i}})^{k^{\alpha_i}r_i}) = i^{r_i}$, which recovers $i^{s_i}$.  

To see that $\psi \circ \varphi (\pi) = \pi$ for any $\pi \in \mathcal{O}_{j,k,b}(n)$, we first observe that if $i$ is a part in $\pi$ divisible by $kb$ with multiplicity $s_i$, then $\varphi(i^{s_i}) = (\frac{i}{k})^{ks_i}$, and then $ks_i=0+ks_i$, so $\psi((\frac{i}{k})^{ks_i}) = i^{s_i}$.  Next we consider any part $i$ for which $kb\nmid i$ but $b\mid i$.  Then $\varphi(i^{s_i}) = i^{r_{i}}, (k i)^{u_{i1}}, (k^2 i)^{u_{i2}}, \ldots, (k^w i)^{u_{iw}}$.  Since $0\leq r_i, u_{ij} \leq k-1$, and $k^ji=bk^jm$ with $k\nmid m$, it follows that $\psi(i^{r_i})=i^{r_i}$, and $\psi((k^ji)^{u{ij}}) = i^{k^ju_{ij}}$ for each $u_{ij}$.  So together, $\psi$ recovers $i^{r_i + u_{i1}k + u_{i2}k^2 + \cdots + u_{iw}k^w} = i^{s_i}$.

We have shown that $\varphi$ and $\psi$ are inverses, and thus $|\mathcal{O}_{j,k,b}(n)| = |\mathcal{D}_{j,k,b}(n)|$ as desired.
\end{proof}

As an example, consider the partition $\pi=4^56^112^718^824^936^1$.  We see that $\pi\in D_{3,2,6}(506)$ as there are exactly three parts, $12$, $18$, and $24$, which are both divisible by $6$ and occur at least twice.  To apply $\psi$ we first observe that $\alpha_6=\alpha_{18}=0$, $\alpha_{12}=\alpha_{36}=1$, and $\alpha_{24}=2$.  Thus under $\psi$, $6^1\mapsto 6^1$; $12^7\mapsto 24^3, 6^2$; $18^8\mapsto 36^4$; $24^9\mapsto 48^4, 6^4$; and $36^1\mapsto 18^2$.  We obtain
\[
\psi(\pi) = 4^56^718^224^336^448^4,
\]
which is in $O_{3,2,6}(506)$ as there are exactly three parts, $24$, $36$, and $48$, which are divisible by $12$.  Going the other direction, for parts divisible by $12$ we observe that $24^3\mapsto 12^6$; $36^4\mapsto 18^8$; $48^4\mapsto 24^8$.  Moreover, writing $2$ and $7$ base $2$ gives that $18^2\mapsto 36^1$; $6^7\mapsto 6^1, 12^1, 24^1$.  Together this gives that $\varphi(\psi(\pi))=\pi$.

In the following table, we show the complete bijective correspondence by row for the eight partitions counted by $O_{3,2,2}(29)$ and $D_{3,2,2}(29)$.  \\


\begingroup
\renewcommand{\arraystretch}{1.2}
\[
        \begin{tabular}{| c  c  c |}
            \hline
            $\mathcal{O}_{3,2,2}(29)$ &  & $\mathcal{D}_{3,2,2}(29)$ \\
            \hline
            $1^14^18^116^1$ & $\longleftrightarrow$ & $1^12^24^28^2$ \\
            $4^15^18^112^1$ & $\longleftrightarrow$ & $2^24^25^16^2$ \\
            $1^14^28^112^1$ & $\longleftrightarrow$ & $1^12^44^26^2$ \\
            $2^13^14^18^112^1$ & $\longleftrightarrow$ & $2^33^14^26^2$ \\
            $1^23^14^18^112^1$ & $\longleftrightarrow$ & $1^22^23^14^26^2$ \\
            $1^12^24^18^112^1$ & $\longleftrightarrow$ & $1^12^24^36^2$ \\
            $1^32^14^18^112^1$ & $\longleftrightarrow$ & $1^32^34^26^2$ \\
            $1^54^18^112^1$ & $\longleftrightarrow$ & $1^52^24^26^2$ \\
            \hline
        \end{tabular}\medskip
\]
\endgroup

\section{A Beck-Type Companion Identity for Theorem \ref{OD24}} \label{sec:Beck}

In this section we prove Theorem \ref{Ojkb/Djkb-beck} in two ways using $q$-series.  The first is a direct proof using \eqref{eq:Ozwq} and \eqref{eq:Dzwq}, and the second is via a modular refinement using \eqref{eq:Ozwq-t}, \eqref{eq:Ozwq-0}, and \eqref{eq:Dzwq-t}, following an approach of Ballantine and Welch \cite[Thm. 4]{ballantinewelch2021becktype}.  We note that Ballantine and Welch also gave a combinatorial proof of \cite[Thm. 4]{ballantinewelch2021becktype} using a bijection of Xiong and Keith \cite{xiongkeith2020eulers}.  It would be interesting to see if their approach could be extended to provide a combinatorial proof of Theorem \ref{Ojkb/Djkb-beck} as well.

In general, given a partition function $a(m,n)$ which counts partitions of $n$ satisfying some condition $(\star)$ into $m$ parts satisfying some condition $(\heartsuit)$, we can access the total number of parts in partitions counted by $a(m,n)$ for all $m$ using a derivative.  In particular, if 
\[
G(w,q) = \sum_{n,m\geq 0} a(m,n)w^mq^n,
\]
where $q$ tracks the size of the partitions counted by $a(m,n)$ and $w$ tracks the number of parts satisfying condition $(\heartsuit)$, then
\begin{equation}\label{eq:ma}
\frac{\partial}{\partial w} G(w,q) \Bigr|_{w=1} = \sum_{n\geq 0} \sum_{m \geq 0} m \cdot a(m,n) q^n, 
\end{equation}
where the coefficient $\sum_{m \geq 0} m \, a(m,n)$ of $q^n$ in \eqref{eq:ma} is the total number of parts satisfying $(\heartsuit)$ in partitions counted by $a(m,n)$ for all $m$.  As our generating functions are all infinite products, we also observe
\begin{equation}\label{eq:prodrule}
\frac{\partial}{\partial w} \prod_{i\geq 1} f_i(w) \Bigr|_{w=1} = \Big( \prod_{i\geq 1} f_i(1) \Big) \sum_{i\geq 1} \frac{\frac{df_i}{dw} |_{w=1}}{f_i(1)}. 
\end{equation}

\begin{proof}[Proof of Theorem \ref{Ojkb/Djkb-beck} via generating functions.]
From \eqref{eq:ma}, it follows that $E(O_{j,k,b}(n),D_{j,k,b}(n))$ is the coefficient of $q^n$ in 
\begin{equation}\label{eq:ODexcess}
\frac{\partial}{\partial w} \sum_{n,j,m \geq 0} O_{j,k,b}(m,n)z^j w^m q^n \Bigr|_{w=1} - \frac{\partial}{\partial w} \sum_{n,j,m \geq 0} D_{j,k,b}(m,n)z^j w^m q^n \Bigr|_{w=1}.
\end{equation}
Setting $f_i(w) = \frac{1-(1-z)wq^{kbi}}{1-wq^i}$ in \eqref{eq:prodrule}, we have from \eqref{eq:Ozwq} that 
\[
\frac{\partial}{\partial w} \sum_{n,j,m \geq 0} O_{j,k,b}(m,n)z^j w^m q^n \Bigr|_{w=1} = \frac{((1-z)q^{kb};q^{kb})_\infty}{(q;q)_\infty} \cdot \sum_{i\geq 1} \frac{q^i-(1-z)q^{kbi}}{(1-q^i)(1-(1-z)q^{kbi})}.
\]
Similarly, setting $f_i(w) = \frac{1-(1-z)w^kq^{kbi}}{1-wq^i}$ in \eqref{eq:prodrule}, we have from \eqref{eq:Dzwq} that
\begin{multline*}
\frac{\partial}{\partial w} \sum_{n,j,m \geq 0} D_{j,k,b}(m,n)z^j w^m q^n \Bigr|_{w=1} \\
= \frac{((1-z)q^{kb};q^{kb})_\infty}{(q;q)_\infty} \cdot \sum_{i\geq 1} \frac{q^i-k(1-z)q^{kbi}+(k-1)(1-z)q^{(kb+1)i}}{(1-q^i)(1-(1-z)q^{kbi})}.
\end{multline*}
Then using \eqref{(j+1)Ogen}, \eqref{eq:ODexcess} becomes
\begin{multline*}
\frac{((1-z)q^{kb};q^{kb})_\infty}{(q;q)_\infty} \cdot \sum_{i\geq 1} \frac{(k-1)(1-z)q^{kbi}-(k-1)(1-z)q^{(kb+1)i}}{(1-q^i)(1-(1-z)q^{kbi})} \\
= (k-1)(1-z) \sum_{n,j \geq 0} (j+1)O_{j+1,k,b}(n)z^j q^n,
\end{multline*}
and so
\begin{multline}\label{eq:1-zconv}
\frac{\partial}{\partial w} \sum_{n,j,m \geq 0} O_{j,k,b}(m,n)z^j w^m q^n \Bigr|_{w=1} - \frac{\partial}{\partial w} \sum_{n,j,m \geq 0} D_{j,k,b}(m,n)z^j w^m q^n \Bigr|_{w=1} \\
= (k-1)(1-z) \sum_{n,j \geq 0} (j+1)O_{j+1,k,b}(n)z^j q^n = (k-1) \sum_{n,j \geq 0}((j+1)O_{j+1,k,b}(n) - jO_{j,k,b}(n))z^jq^n,
\end{multline}
and the desired result follows from comparing the coefficients of $z^jq^n$.
\end{proof}

We next show that Theorem \ref{Ojkb/Djkb-beck} can also be proved via a modular refinement, modifying the method of Ballantine and Welch \cite{ballantinewelch2021becktype}.  Recall that for a partition $\pi$, we define the residual multiplicity modulo $k$ of the part $i$ in $\pi$, denoted by $r_i(\pi)$, by \eqref{resmultdef}. 
Further, recall $\ell_{kb,tb}(\pi)$ is the total number of parts in $\pi$ that are $\equiv tb \pmod{kb}$, and define $\bar{\ell}_{b,0,t}(\pi)$ to be the number of different parts in $\pi$ that are $\equiv 0 \pmod{b}$ and satisfy $r_i(\pi)\geq t$.  Define
\begin{equation}\label{Ejkbtdef}
E_{j,k,b,t}(n) = \sum_{\pi \in \mathcal{O}_{j,k,b}(n)} (\ell_{kb,tb}(\pi) - \ell_{kb,0}(\pi)) - \sum_{\pi \in \mathcal{D}_{j,k,b}(n)} \bar{\ell}_{b,0,t}(\pi).
\end{equation}
We prove the following.

\begin{proposition}\label{ojkb/djkb-beck-refinement}
For integers $n,j \geq 0$, $k \geq 2$, $b \geq 1$ and $1 \leq t \leq k-1$, we have 
\[
E_{j,k,b,t}(n) = (j+1)O_{j+1,k,b}(n) - jO_{j,k,b}(n) = (j+1)D_{j+1,k,b}(n) - jD_{j,k,b}(n).
\]
\end{proposition}

\begin{proof}
We first observe that for any $0 \leq t \leq k-1$,
\begin{align} \label{O-tcount}
\sum_{\pi \in \mathcal{O}_{j,k,b}(n)} \ell_{kb,tb}(\pi) & = \sum_{m\geq 0} m\cdot O_{j,k,b,t}(m,n), \\
\sum_{\pi \in \mathcal{D}_{j,k,b}(n)} \bar{\ell}_{b,0,t}(\pi) & = \sum_{m\geq 0} m\cdot D_{j,k,b,t}(m,n). \nonumber
\end{align}
Hence it follows from \eqref{eq:ma} that
\begin{multline*}
\sum_{n,j\geq0}E_{j,k,b,t}(n)z^jq^n = 
\frac{\partial}{\partial w} \sum_{n,j,m \geq 0} (O_{j,k,b,t}(m,n) - O_{j,k,b,0}(m,n) - D_{j,k,b,t}(m,n))z^j w^m q^n \Bigr|_{w=1}.
\end{multline*}
Thus using \eqref{(j+1)Ogen} and \eqref{eq:1-zconv} it suffices to prove that
\begin{multline}\label{eq:target}
\frac{\partial}{\partial w} \sum_{n,j,m \geq 0} (O_{j,k,b,t}(m,n) - O_{j,k,b,0}(m,n) - D_{j,k,b,t}(m,n))z^j w^m q^n \Bigr|_{w=1} \\
= (1-z) \frac{((1-z)q^{kb};q^{kb})_\infty}{(q;q)_\infty} \sum_{i\geq 1} \frac{q^{kbi}}{1-(1-z)q^{kbi}}.
\end{multline}

Let $1\leq t \leq k-1$.  Setting $f_i(w) = (1-wq^{tb+kbi})^{-1}$ in \eqref{eq:prodrule}, we obtain from \eqref{eq:Ozwq-t} that 
\[
\frac{\partial}{\partial w} \sum_{n,j,m \geq 0} O_{j,k,b,t}(m,n) z^jw^mq^n \Bigr|_{w=1} = \frac{((1-z)q^{kb};q^{kb})_\infty}{(q;q)_\infty} \sum_{i\geq 0}\frac{q^{tb+kbi}}{1-q^{tb+kbi}}.
\]
Expanding $1/(1-q^{tb+kbi})$ as a series yields that
\[
\sum_{i\geq 0}\frac{q^{tb+kbi}}{1-q^{tb+kbi}} = \sum_{i\geq 0} \sum_{j\geq 1} q^{(tb+kbi)j} = \sum_{j\geq 1} q^{tbj} \sum_{i\geq 0} q^{kbij} = \sum_{j\geq 1} \frac{q^{tbj}}{1-q^{kbi}}.
\]
Thus in fact,
\begin{equation}\label{O-tderiv}
\frac{\partial}{\partial w} \sum_{n,j,m \geq 0} O_{j,k,b,t}(m,n) z^jw^mq^n \Bigr|_{w=1} = \frac{((1-z)q^{kb};q^{kb})_\infty}{(q;q)_\infty} \sum_{i\geq 1} \frac{q^{tbi}}{1-q^{kbi}}.
\end{equation}

Setting $f_i(w) = \frac{1-(1-z)wq^{kbi}}{1-wq^{kbi}}$ in \eqref{eq:prodrule}, we obtain directly from \eqref{eq:Ozwq-0} that
\begin{equation}\label{O-0deriv}
\frac{\partial}{\partial w} \sum_{n,j,m \geq 0} O_{j,k,b,0}(m,n) z^jw^mq^n \Bigr|_{w=1} = \frac{((1-z)q^{kb};q^{kb})_\infty}{(q;q)_\infty} \sum_{i\geq 1}\frac{zq^{kbi}}{(1-q^{kbi})(1-(1-z)q^{kbi})},
\end{equation}
and setting $f_i(w) = 1-q^{tbi}+w(q^{tbi}-q^{kbi})$ in \eqref{eq:prodrule}, we obtain from \eqref{eq:Dzwq-t} that
\begin{equation}\label{D-tderiv}
\frac{\partial}{\partial w} \sum_{n,j,m \geq 0} D_{j,k,b,t}(m,n) z^jw^mq^n \Bigr|_{w=1} = \frac{((1-z)q^{kb};q^{kb})_\infty}{(q;q)_\infty} \sum_{i\geq 1}\frac{q^{tbi}-q^{kbi}}{1-q^{kbi}}.
\end{equation}

Subtracting \eqref{O-0deriv} and \eqref{D-tderiv} from \eqref{O-tderiv} gives
\begin{multline*}
\frac{((1-z)q^{kb};q^{kb})_\infty}{(q;q)_\infty} \sum_{i\geq 1} \left( \frac{q^{kbi}}{1-q^{kbi}} - \frac{zq^{kbi}}{(1-q^{kbi})(1-(1-z)q^{kbi})} \right) \\
= \frac{((1-z)q^{kb};q^{kb})_\infty}{(q;q)_\infty} \sum_{i\geq 1} \frac{(1-z)q^{kbi}}{1-(1-z)q^{kbi}},
\end{multline*}
which establishes \eqref{eq:target}.
\end{proof}

We now prove Theorem \ref{Ojkb/Djkb-beck} using Proposition \ref{ojkb/djkb-beck-refinement}. 

\begin{proof}[Proof of Theorem \ref{Ojkb/Djkb-beck} via a modular refinement]
Let $k \geq 2$.  Note that we have seen in our combinatorial proof of Theorem \ref{OD24} that the total number of parts $\not\equiv 0 \pmod{b}$ in the partitions counted by $O_{j,k,b}(n)$ and $D_{j,k,b}(n)$ are equal, as they remain fixed under the bijection.  Thus,
\[
E(O_{j,k,b}(n),D_{j,k,b}(n))=\sum_{\pi \in \mathcal{O}_{j,k,b}(n)} \ell(\pi) - \!\!\! \sum_{\pi \in \mathcal{D}_{j,k,b}(n)} \ell(\pi)=\sum_{\pi \in \mathcal{O}_{j,k,b}(n)} \ell_{b,0}(\pi) - \!\!\! \sum_{\pi \in \mathcal{D}_{j,k,b}(n)} \ell_{b,0}(\pi).
\]

From Proposition \ref{ojkb/djkb-beck-refinement}, we see that for $1\leq t \leq k-1$, $E_{j,k,b,t}(n)$ doesn't depend on $t$, so $\sum_{t=1}^{k-1}E_{j,k,b,t}(n) = (k-1)E_{j,k,b,t}(n)$ and it suffices to show that 
\begin{equation}\label{eq:reftarget1}
\sum_{t=1}^{k-1}E_{j,k,b,t}(n) = \sum_{\pi \in \mathcal{O}_{j,k,b}(n)} \ell_{b,0}(\pi) - \!\!\! \sum_{\pi \in \mathcal{D}_{j,k,b}(n)} \ell_{b,0}(\pi).
\end{equation}

Observe that $\sum_{t=1}^{k-1}\ell_{kb,tb}(\pi) = \ell_{b,0}(\pi) - \ell_{kb,0}(\pi)$.  Thus from \eqref{Ejkbtdef}, we obtain
\[
\sum_{t=1}^{k-1}E_{j,k,b,t}(n) 
= \!\!\! \sum_{\pi \in \mathcal{O}_{j,k,b}(n)}\ell_{b,0}(\pi)-k \!\!\!  \sum_{\pi \in \mathcal{O}_{j,k,b}(n)}\ell_{kb,0}(\pi) - \!\!\! \sum_{\pi \in \mathcal{D}_{j,k,b}(n)} \sum_{t=1}^{k-1} \bar{\ell}_{b,0,t}(\pi),
\]
and we can see from \eqref{eq:reftarget1} that it now suffices to show 
\begin{equation}\label{eq:reftarget2}
\sum_{\pi \in \mathcal{D}_{j,k,b}(n)} \ell_{b,0}(\pi) = k \!\!\!  \sum_{\pi \in \mathcal{O}_{j,k,b}(n)}\ell_{kb,0}(\pi) + \!\!\! \sum_{\pi \in \mathcal{D}_{j,k,b}(n)} \sum_{t=1}^{k-1} \bar{\ell}_{b,0,t}(\pi).
\end{equation}

Recall that $\bar{\ell}_{b,0,t}(\pi)$ counts the number of different parts of the form $ib$ in $\pi$ such that $r_{ib}(\pi)\geq t$.  Thus for a given partition $\pi \in \mathcal{D}_{j,k,b}(n)$, 
\begin{equation}\label{ellbar}
\sum_{t=1}^{k-1} \bar{\ell}_{b,0,t}(\pi) = \sum_{i\geq 0} r_{ib}(\pi),
\end{equation}
as each $ib$ will contribute $1$ to the sum on the left side for exactly $r_{ib}(\pi)$ values of $t$.  Furthermore, using \eqref{resmultdef}, we have that
\[
\ell_{b,0}(\pi) = \sum_{i\geq 0} s_{ib}(\pi) = k \sum_{i\geq 0} u_{ib}(\pi) +  \sum_{i\geq 0} r_{ib}(\pi).
\]
Thus by \eqref{ellbar},
\[
\sum_{\pi \in \mathcal{D}_{j,k,b}(n)} \ell_{b,0}(\pi) = k \!\!\! \sum_{\pi \in \mathcal{D}_{j,k,b}(n)} \sum_{i\geq 0} u_{ib}(\pi)  + \!\!\! \sum_{\pi \in \mathcal{D}_{j,k,b}(n)}\sum_{t=1}^{k-1} \bar{\ell}_{b,0,t}(\pi),
\]
so now we have reduced \eqref{eq:reftarget2} to showing that
\begin{equation}\label{target3}
\sum_{\pi \in \mathcal{D}_{j,k,b}(n)} \sum_{i\geq 0} u_{ib}(\pi) =  \sum_{\pi \in \mathcal{O}_{j,k,b}(n)}\ell_{kb,0}(\pi).
\end{equation}

Let $\overline{D}_{j,k,b}(m,n)$ be the number of partitions counted by ${D}_{j,k,b}(n)$ such that $u_{ib}(\pi) = m$.  Then it follows that
\[
\sum_{\pi \in \mathcal{D}_{j,k,b}(n)} \sum_{i\geq 0} u_{ib}(\pi) = \sum_{m\geq 0} m \cdot \overline{D}_{j,k,b}(m,n),
\]
and as we saw in \eqref{O-tcount}, 
\[
\sum_{\pi \in \mathcal{O}_{j,k,b}(n)} \ell_{kb,0}(\pi) = \sum_{m\geq 0} m \cdot O_{j,k,b,0}(m,n).
\]
Thus if we show that
\begin{equation}\label{target4}
\sum_{n,j,m\geq 0} O_{j,k,b,0}(m,n) z^j w^m q^n = \sum_{n,j,m\geq 0} \overline{D}_{j,k,b}(m,n) z^j w^m q^n,
\end{equation}
then \eqref{target3} follows from \eqref{eq:ma}.  From \eqref{eq:Ozwq-0} we have 
\[
\sum_{n,j,m \geq 0} O_{j,k,b,0}(m,n)z^j w^m q^n =  \frac{((1-z)wq^{kb};q^{kb})_\infty(q^{kb};q^{kb})_\infty }{(wq^{kb};q^{kb})_\infty(q;q)_\infty},
\]
so we need to interpret this in terms of partitions counted by $D_{j,k,b}(n)$.  Observe that 
\begin{multline*}
\frac{((1-z)wq^{kb};q^{kb})_\infty (q^{kb};q^{kb})_\infty}{(wq^{kb};q^{kb})_\infty(q;q)_\infty} 
= \frac{(q^{b};q^{b})_\infty}{(q;q)_\infty} \cdot \frac{(q^{kb};q^{kb})_\infty}{(q^{b};q^{b})_\infty} \cdot \frac{((1-z)wq^{kb};q^{kb})_\infty}{(wq^{kb};q^{kb})_\infty} \\
= \frac{(q^{b};q^{b})_\infty}{(q;q)_\infty} \prod_{i\geq 1} (1 + q^{bi} + q^{2bi} + \cdots + q^{(k-1)bi}) (1 + zwq^{k(bi)} + zw^2q^{2k(bi)} + zw^3 q^{3k(bi)} + \cdots ).
\end{multline*}
The product on the far left generates parts not divisible by $b$ which are unrestricted.  For the product in the center, we can interpret $q^{nbi}$ for $0\leq n\leq k-1$ as generating $n$ copies of part $bi$, and in the product on the right we can interpret $q^{nk(bi)}$ for $n\geq 0$ as generating $nk$ copies of the part $bi$.  Hence, a pair of choices, one from the product in the center and one from the product on the right, corresponds to a choice of multiplicity of the part $bi$ in a partition $\pi$ in terms of $r_{bi}(\pi)$ and $u_{bi}(\pi)$.  It follows that $z$ will track the number of different parts $bi$ that appear $\geq k$ times, and $w$ will track $\sum_{i\geq 1} u_{ib}(\pi)$ as desired.  Thus \eqref{target4} is obtained which finishes the proof.
\end{proof}

\end{document}